\documentclass[12pt]{article}

\usepackage{amsfonts}
\usepackage{amsmath}
\usepackage{amsthm}
\usepackage{amssymb}
\usepackage{graphicx}

\usepackage[dvipsnames]{xcolor}
\usepackage{todonotes}

\def\be{\begin{equation}}
\def\ee{\end{equation}}
\def\l{\langle}
\def\r{\rangle}
\def\p{\parallel}
\def\R{I\!\!R}
\def\cR{{\cal R}}
\def\cN{{\cal N}}
\def\ba{\bar{A}}
\def\rank{\mathrm{rank}}
\def\dim{\mathrm{dim}}
\def\diag{\mathrm{diag}}

\newtheorem{remark}{Remark}
\newtheorem{lemma}{Lemma}

\newtheorem{proposition}{Proposition}

\begin{document}

\begin{center}

\vspace*{0.2cm}
{\large \bf On numerical solution of full rank linear systems}

\centerline{\today}

\vspace*{0.2cm}
\centerline
{A. Dumitra\c sc$^\spadesuit$, Ph. Leleux$^{++}$, C. Popa$^{+}$, D. Ruiz$^{++}$ and U. R\" ude$^\spadesuit$,$^{++}$ }
\end{center}

\begin{center}
{\small 
$^\spadesuit$FAU Erlangen-Nurnberg, Germany\\
$^{++}$CERFACS Toulouse, France\\
$^+$Ovidius University of Constanta, Romania 
}
\end{center}
\footnotetext{For C. Popa this paper was partially supported by the DAAD Grant nr. 57440915/2019}

\begin{abstract}
Matrices can be augmented by adding additional columns
such that a partitioning of the matrix in 
blocks of rows defines mutually orthogonal subspaces.
This augmented system can then be solved efficiently 
by a sum of projections onto these subspaces. The equivalence
to the original linear system is ensured by
adding additional rows to the matrix in a specific form.
The resulting solution method is known as the augmented block
Cimmino method. Here this method is extended to full rank
underdetermined systems and to overdetermined systems.
In the latter case, rows of the matrix, not columns,
must be suitably augmented. 
The article presents an analysis of these methods.
\end{abstract}

\vspace*{0.2cm}

\vspace*{0.cm}

\section{Introduction}
Designing efficient numerical solutions for large, sparse, 
ill-conditioned linear systems of equations 
remains a challenge for 
scientific computing, where larger and larger systems must be solved.
Unfortunately, there are no \textit{universal} solvers, 
such as Gaussian elimination (as a direct method) or 
Kaczmarz (as an iterative method) 
that can solve every square nonsingular system of linear equations
efficiently without supplementary assumptions. 
Usually, for obtaining/designing efficient solvers,
we must exploit the specific information, such as the structure or special properties,
of the problem matrix.
Often these properties are directly related to the concrete/real world problem
that we want to solve.
This is the setting of this article. 
We propose and theoretically analyze such a specific solver.
The method is based on the construction of an augmented problem,
adding rows or columns to the original matrix. 
In specific cases, i.e.~for specific classes of problems, these augmented systems
can be solved with  a well designed parallel algorithm.
In such cases, solving the augmented larger system can be an efficient method
to produce a solution of the original system.

The method of interest in our article was first considered 
in \cite{duffsisc} for square nonsingular systems of linear equations. 
In the present 
article we extend and develop 
these methods for full rank over- and underdetermined systems of linear equations.
In particular, we 
provide a complete theoretical analysis of the proposed procedure. 
Efficient numerical implementations, 
as well as computational considerations on specific classes of problems are 
described in \cite{duffsisc}, \cite{adumi}, and \cite{adumi1}.

The paper is organized as follows. In section \ref{prelim} we introduce the basic notation and definitions necessary in the rest of the paper, and briefly describe the main ideas of the procedure proposed in \cite{duffsisc}. In section \ref{consist} we adapt and extend the results from \cite{duffsisc} to underdetermined full row rank linear systems (which are always consistent). In section \ref{inconsist} we adapt and develop the construction from section \ref{consist} to the case of overdetermined full column rank linear systems. These are usually inconsistent and we must reformulate them as linear least squares problems. 
This new aspect makes the theoretical analysis more elaborated than in the consistent case.
The paper finishes with final comments on open problems and further research directions in the field.
%
%
\section{Preliminaries}
\label{prelim}
We start  this introductory section presenting the notations and definitions used in the rest of the paper.  By $\l \cdot, \cdot \r, \p \cdot \p$ we will denote the Euclidean scalar product and norm on some space $\R^q$. If $A$ is a real $m \times n$ matrix we will denote by $A^T, a_i, a^j, \rank(A), \cR(A),$ $\cN(A), A^+$ the transpose, $i$-th row, $j$-th column, rank, range, null space and Moore-Penrose pseudoinverse of it.
The vectors $x \in \R^q$ will be considered as column vectors, thus with the above rows and columns, the matrix $A$ can be written as
\be
\label{0-1}
A = \left[ 
		\begin{matrix}
		(a_1)^T \\
		(a_2)^T \\
		 \dots  \\
		(a_m)^T \\
		\end{matrix}
		\right] ~~\textrm{or}~~ A = [a^1 a^2 \dots a^n ].
\ee
If for $1 \leq p < m$, we split the rows indices as
	$1 \leq m_1 < m_2 < \dots < m_p = m$
and the subsets
\be
\label{0-2}
N_1 = \{ 1, \dots, \mu_1 \}, N_2 = \{ \mu_1+1, \dots, \mu_2 \}, \dots, 
N_p = \{ \mu_{p-1}+1, \dots, \mu_p \},
\ee
and define the row blocks $A_1, A_2, \dots, A_p$ of $A$, without overlapping rows as 
\be
\label{0-33p}
A_1 = \left[ 
		\begin{matrix}
		(a_1)^T \\
			 \dots  \\
		(a_{\mu_1})^T \\
		\end{matrix}
		\right], ~~A_2 = \left[ 
		\begin{matrix}
		(a_{\mu_1+1})^T \\		 
		\dots  \\
		(a_{\mu_2})^T \\
		\end{matrix}
		\right], ~~A_p = \left[ 
		\begin{matrix}
		(a_{\mu_{p-1}+1})^T \\		 
		\dots  \\
		(a_{\mu_p})^T \\
		\end{matrix}
		\right],
\ee
then $A$ and $A^T$ will be written as 
\be
\label{0-4p}
 A = \left[ 
		\begin{matrix}
		A_1  \\
		\dots \\
		A_p
		\end{matrix}
		\right], ~~\textrm{or}~~ A = [A_1^T  A_2^T \dots A_p^T ].
\ee
If $m=n$ and $A$ is invertible, $A^{-1}$ will denote its inverse.
The orthogonal projector onto a vector subspace $S \in \R^q$
will be written as $P_S$ and the dimension of $S$ as 
$\dim(S)$. $I_q, ~O_q$ will stand for the unit,
respectively zero matrix of order $q$,
and $D = \diag(\delta_1, \delta_2, \dots, \delta_n)$ 
will denote the diagonal matrix
$$
D =
\left[ 
		\begin{matrix}
		\delta_1 & 0 & \dots & 0\\
		0 & \delta_2 & \dots & 0 \\
		 & & \ddots & \\
		0 & 0 &\dots &\delta_n \\
		\end{matrix}
		\right] .
$$
If $A: m \times n, b \in \R^m$ and $b \in \cR(A)$, we will denote as 
\be
\label{intro-1}
Ax=b
\ee
the corresponding system of linear equations, by $S(A; b)$ the set of its (classical) solutions  and by $x_{LS}$ the (unique) minimal norm one. 
In the general case, for $b \in \R^m$, the system (\ref{intro-1}) will be formulated in a least squares sense: find $x \in \R^n$ such that 
\be
\label{intro-2}
\p Ax-b \p = \min \{ \p Az-b \p, z \in \R^n \},
\ee
and denote by $LSS(A; b)$ the set of its (least squares) solutions  and by $x_{LS}$ the (unique) minimal norm one. 
We know that (see e.g. \cite{popabook}) in both cases (\ref{intro-1}) and  (\ref{intro-2})
\be
\label{intro-3}
x_{LS} = A^+ b ~~{\rm and}~~ A x_{LS} = P_{\cR(A)}(b).
\ee

In the rest of this section we will briefly 
recapitulate the augmentation procedure 
that was first proposed in \cite{duffsisc}.
We start from a square nonsingular system of linear equations
\be
\label{al-11}
\tilde{A} \tilde{x} = \tilde{b},
\ee
$\tilde{A}: m \times m, ~\tilde{b} \in \R^m$ and reorder it as 
\be
\label{al-12}
A x = b,
\ee
where 
\be
\label{1}
A =  P \tilde{A} Q =  \left[ 
		\begin{matrix}
		A_1  \\
		\dots \\
		A_p
		\end{matrix}
		\right], ~b=P \tilde{b}, ~x = Q^T \tilde{x},
\ee
with $A_1, \dots, A_p$ row blocks as in (\ref{0-33p}),
\be
\label{111}
 A_i: m_i \times n, ~~m_1 + \dots + m_p = m,
\ee
and $P, Q: m \times m$ permutation  matrices. 
This reordered system is then augmented to 
\be
\label{al-21}
\left[ 
		\begin{matrix}
		A & \Gamma \\
		0 & I_q \\
		\end{matrix}
		\right] \left[ 
		\begin{matrix}
		x \\
		y \\ 
		\end{matrix}
		\right] = \left[ 
		\begin{matrix}
		b \\
		0 \\ 
		\end{matrix}
		\right].
\ee
Note that the extended matrix   
$\ba: m \times \bar{m}, \bar{m}=m + q, q \geq 1$ has the block structure
\be
\label{3}
\ba = [ A ~~\Gamma ] = 
		 \left[ 
		\begin{matrix}
		\ba_1  \\
		\dots \\
		\ba_p
		\end{matrix}
		\right], ~\Gamma = \left[ 
		\begin{matrix}
		 \Gamma_1  \\
		\dots \\
		\Gamma_p
		\end{matrix} \right ], ~\bar{A}_i = [ A_i  ~~\Gamma_i ], i=1, \dots, p.  
\ee
Here, we have constructed the augmentation blocks
$\Gamma: m \times q, \Gamma_i: m_i \times q$, 
such that the row blocks $\ba_i$ are mutually orthogonal, i.e. 	
\be
\label{202}
	\ba_i \ba^T_j =0, ~~\forall \; i \neq j. 
\ee 
The solution $x$ is preserved by requiring that \textit{the augmentation variables} satisfy
$y=0$.
Unfortunately the last $q$ rows in (\ref{al-21}) that correspond to this condition 
will generally violate the orthogonality relation.
In particular,
the blocks $\bar{A} = [ A ~~\Gamma ]$ and $Y = [ 0 ~~I_q ]: q \times \bar{n}$
are not orthogonal.
Therefore, the authors in \cite{duffsisc} propose the modification
\be
\label{al-22}
\left[ 
		\begin{matrix}
		A & \Gamma \\
		B & S \\
		\end{matrix}
		\right] \left[ 
		\begin{matrix}
		x \\
		y \\ 
		\end{matrix}
		\right] = \left[ 
		\begin{matrix}
		b \\
		f \\ 
		\end{matrix}
		\right], 
\ee
 in which  the blocks 
\be
	\label{al-22p}
	\bar{A} = [ A ~~\Gamma ] ~{\rm and}~ W = [ B ~~S ]
\ee
are constructed to be orthogonal. Then, the (parallel) solution procedure for getting a solution for (\ref{al-11}) through (\ref{al-22}) is the following. 
\begin{itemize}
\item[(1.1)] The minimal norm solution of the system (\ref{al-22}), say
$		\left[ 
		\begin{matrix}
		x^* \\
		y^* \\ 
		\end{matrix}
		\right]
$, is computed through (see e.g.~\cite{horn})
\be \label{al-22p1}
		\left[ 
		\begin{matrix}
		x^* \\
		y^* \\ 
		\end{matrix}
		\right] = \left[ 
		\begin{matrix}
		\bar{A} \\
		W \\ 
		\end{matrix}
		\right]^+ \left[ 
		\begin{matrix}
		b \\
		f \\ 
		\end{matrix}
		\right] .
\ee
\item[(1.2)] Because of the mutual orthogonality of the blocks
$\bar{A}_i, i=1, \dots, p$ and $W$ we get
\be \label{all-22p2}
		\left[ 
		\begin{matrix}
		\bar{A} \\
		W \\ 
		\end{matrix}
		\right]^+ = \left[ 
		\begin{matrix}
		\bar{A}_1^+ ~\bar{A}_2^+ \dots ~\bar{A}_p^+ ~ W^+ \\ 
		\end{matrix}
		\right] .
\ee
Therefore, according to (\ref{al-22p1}) this gives us
\be \label{121}
		\left[ 
		\begin{matrix}
		x^* \\
		y^* \\ 
		\end{matrix}
		\right]  = \bar{A}^+ b + W^+ f = \bar{A}^+ 
			\left[ 
		\begin{matrix}
		b^1 \\
		b^2 \\ 
		\dots \\
		b^p \\
		\end{matrix}
		\right] + W^+ f = 
		\sum_{i=1}^p \bar{A}_i^+ b^i + W^+ f,
\ee
where
$b = \left[ 
		\begin{matrix}
		b^1 \\
		b^2 \\ 
		\dots \\
		b^p \\
		\end{matrix}
		\right]$
is the splitting of the vector $b$ with respect to the partitioning of $A$ in (\ref{1}).
\item[(1.3)] But, for an appropriate choice of the vector $f$, it can be shown that the $x^*$ part of the minimal norm solution in (\ref{121}) becomes a solution of the initial system (\ref{al-11}),
i.e.~$x^*= \tilde{A}^{-1}\tilde{b}$. 
%
\item[(1.4)] The terms of the final sum in (\ref{121}) can be computed in parallel. An efficient implementation of such a computation, can transform this solution procedure into a very efficient solver (see \cite{duffsisc}, ...).
\end{itemize}
\begin{remark} \label{rerlg2}
		(i) For reordering $\tilde{A}$ as in (\ref{1}) the authors propose in \cite{duffsisc} (see also \cite{adumi} and \cite{adumi1}) the Cuthill-McKee algorithm from \cite{reid2006reducing}. This ensures a block bidiagonal or tridiagonal structure of $A$.
		
		(ii) For the augmentation of $A$ to $\bar{A}$ such that the augmented row blocks $\bar{A}_i$ are mutually orthogonal (see (\ref{202})) several procedures are proposed in \cite{duffsisc} (see also \cite{adumi} and \cite{adumi1}). 
\end{remark}
%
%
\section{Full row rank underdetermined systems}
\label{consist}
In this section we will 
recapitulate the results from \cite{duffsisc} and extend them in the case of full row rank underdetermined systems. 
Therefore the matrix $\tilde{A}$ in (\ref{al-11}) will be $m \times n$, with $m \leq n$ and $\rank(\tilde{A}) = m.$ {Hence, for any $\tilde{b} \in \R^m$ the system with $\tilde{A}$ and $\tilde{b}$ (of the form (\ref{al-11})) will be consistent, and so will be the system obtained after reordering $\tilde{A}$ (of the form (\ref{al-12}))}. The value $\bar{m}$ appearing as the second dimension of the matrix $\bar{A}$ in (\ref{3}) will be denoted by $\bar{n}$, with $\bar{n} = n + q, q \geq 1.$ We will also consider the projection operator 	$P = P_{\cR(\ba^T)}$, which under the mutual  orthogonality hypothesis (\ref{202}) is given by (see e.g.  \cite{horn})
\be \label{4}
	P = \bar{A}^+ \bar{A} = \sum_{i=1}^p P_{\cR(\ba_i^T)} ~\textrm{with}~ 
	P_{\cR(\ba_i^T)} =  \ba^+_i \ba_i.
\ee
\begin{proposition}
\label{lem1}
(i) If we set
\be \label{5}
	W = Y (I -P), ~{\rm where}~ Y = [ 0 ~I_q ], 
\ee
then the row block $W$ from (\ref{al-22p}) 
is orthogonal 
to $\ba$, hence to each row block $\ba_i, i=1, \dots, p$.\\
(ii) We have the equalities
\be
\label{6}
\bar{A}^+ = [ \bar{A}^+_1 \dots \bar{A}^+_p ], 
\left[ 
		\begin{matrix} 
		\ba \\
		W
		\end{matrix}
		\right]^+ = [ \ba^+  ~W^+ ]  
		\ee
		and
		\be
\label{6p}
		W W^T = B B^T + S^2 = S, ~\textrm{where}~  S = Y (I - P ) Y^T: q \times q.
\ee
(iii) Let us suppose that the matrix $S$ from (\ref{6p}) is invertible and  $f$ is given by
\be
\label{5p}
f=-Y \ba^+ b.
\ee
Then, if $x$ is a solution of the system (\ref{al-12}), the vector $\left[ 
		\begin{matrix}
		x \\
		0 \\ 
		\end{matrix}
		\right]$ is a solution of (\ref{al-22}). Conversely, if 
$\left[ 
		\begin{matrix}
		x \\
		y \\ 
		\end{matrix}
		\right]$ is the minimal norm solution of the system (\ref{al-22}), then $y=0$ and $x$ is a solution of (\ref{al-12}).\\
		(iv) The vector $\tilde{x}$ is a solution of the system (\ref{al-11}) if and only if the vector $Q^T \tilde{x}$ is a solution of the system (\ref{al-12}), where $Q$ is the permutation matrix from (\ref{1}).
\end{proposition}
\begin{proof}
	The proofs for the conclusions $(i)$, $(ii)$ and $(iv)$ are given in \cite{duffsisc}. 
We will present here only the proof of $(iii)$ which is {different and much more detailed  that the one in \cite{duffsisc}}.
%
Thus, if $x$ is a solution of  (\ref{al-12}) we have (see  (\ref{6p}), (\ref{5p}))
\begin{eqnarray*}
\left[ 
		\begin{matrix}
		A & \Gamma \\
		B & S \\
		\end{matrix}
		\right] \left[ 
		\begin{matrix}
		x \\
		0 \\ 
		\end{matrix}
		\right] & = & \left[ 
		\begin{matrix}
		Ax \\
		W \left[ 
		\begin{matrix}
		x \\
		0 \\ 
		\end{matrix}
		\right] \\ 
		\end{matrix}
		\right] = \left[ 
		\begin{matrix}
		b \\
		Y(I-P) \left[ 
		\begin{matrix}
		x \\
		0 \\ 
		\end{matrix}
		\right] \\ 
		\end{matrix}
		\right]  \\[0.5ex]
		& = & \left[ 
		\begin{matrix}
		b \\
		-Y \bar{A}^+ \bar{A} \left[ 
		\begin{matrix}
		x \\
		0 \\ 
		\end{matrix}
		\right] \\ 
		\end{matrix}
		\right] = \left[ 
		\begin{matrix}
		b \\
		-Y \bar{A}^+ b \\ 
		\end{matrix}
		\right] = \left[ 
		\begin{matrix}
		b\\
		f \\ 
		\end{matrix}
		\right],
\end{eqnarray*}
with $f$ from (\ref{5p}), which completes the first part of the proof.

Let now $\left[ 
		\begin{matrix}
		x \\
		y
		\end{matrix}
		\right]$ 
be the minimal norm solution of (\ref{al-22}) with $f$ from (\ref{5p}).
Hence (see  \cite{popabook}, (\ref{intro-3}) and (\ref{6})) 
\be
	\label{8}
		\left[ 
		\begin{matrix}
		x \\
		y
		\end{matrix}
		\right] = \left[ 
		\begin{matrix}
		\ba \\
		W
		\end{matrix}
		\right]^+ \left[ 
		\begin{matrix}
		b \\
		f
		\end{matrix}
		\right]  = [ \ba^+  ~W^+ ] \left[ 
		\begin{matrix}
		b \\
		f
		\end{matrix}
		\right] =  \ba^+ b + W^+ f.
\ee
According to our hypothesis on the invertibility of the matrix $S$, and the second equality in (\ref{6p}) we conclude that $W^T$ has full column rank, 
therefore (see again \cite{popabook}) 
\be
	\label{erlg}
	W^+ = W^T (W W^T)^{-1} = (I-P) Y^T S^{-1} = W^T S^{-1}.
\ee
Then
$$
	W^+ f = W^T S^{-1} f = \left[ 
		\begin{matrix}
		B^T \\
		S^T \\ 
		\end{matrix}
		\right] S^{-1} f = \left[ 
		\begin{matrix}
		B^T S^{-1} f\\
		f \\ 
		\end{matrix}
		\right],
$$
and from (\ref{8}) and (\ref{5}) we obtain ($\bar{n} = n + q$;  see (\ref{3}))
$$
	\ba^+ b + W^+ f = \left[
	\begin{matrix}
	I_n & 0\\
	0 & I_q \\
	\end{matrix}
	\right ]
	\ba^+ b  + W^+ f= 
	\left[ 
		\begin{matrix}
		[I_n ~0] \ba^+ b\\
		[0 ~I_q] \ba^+ b\\ 
		\end{matrix}
		\right] + \left[ 
		\begin{matrix}
		B^T S^{-1} f\\
		f \\ 
		\end{matrix}
		\right]   = 
$$
\be \label{9}
	\left[ 
		\begin{matrix}
		[I_n ~0] \ba^+ b\\
		Y \ba^+ b \\ 
		\end{matrix}
		\right] + \left[ 
		\begin{matrix}
		B^T S^{-1} f\\
		-Y \ba^+ b \\ 
		\end{matrix}
		\right] = \left[ 
		\begin{matrix}
		[I_n ~0] \ba^+ b + B^T S^{-1} f\\
		0 \\ 
		\end{matrix}
		\right] .
\ee
(\ref{8}) and (\ref{9}) yield $y=0$,  hence 
the minimal norm solution of the (consistent) system (\ref{al-22})  has the form 
$\left[ 
	\begin{matrix}
		x \\
		0 \\ 
	\end{matrix}
\right]$, 
with $x$ from (\ref{9}) (first component of the last vector).
In particular we have  
$$	\left[ 
		\begin{matrix}
		A & \Gamma \\
		B & S \\
		\end{matrix}
		\right] \left[ 
		\begin{matrix}
		x \\
		y \\ 
		\end{matrix}
		\right] = \left[ 
		\begin{matrix}
		A & \Gamma \\
		B & S \\
		\end{matrix}
		\right] \left[ 
		\begin{matrix}
		x \\
		 0 \\ 
		\end{matrix}
		\right] = 
		\left[ 
		\begin{matrix}
		A x \\
		B x \\ 
		\end{matrix}
		\right] = \left[ 
		\begin{matrix}
		b \\
		f \\ 
		\end{matrix}
		\right],
$$
i.e.~$b = Ax$ which completes the proof.
\end{proof}
\begin{remark}
	\label{rerlg}
The equalities in (\ref{erlg}) tell us that $W^+ f$ in (\ref{121}) will be computed as 
\be
	\label{erlg1}
	W^+ f = W^T S^{-1} f.
\ee
Therefore, the minimal norm solution, 
$\left[ 
	\begin{matrix}
		x^* \\
		y^* \\ 
	\end{matrix}
\right]$
of the system (\ref{al-22}) will (finally) be computed (in parallel) as (see (\ref{121}) and (\ref{5p})) 
\be
\label{erlg2}
	\left[ 
		\begin{matrix}
		x^* \\
		y^* \\ 
		\end{matrix}
	\right]  =  
	\sum_{i=1}^p \bar{A}_i^+ b^i - (I - P) Y^T S^{-1} Y \sum_{i=1}^p \bar{A}_i^+ b^i.
\ee
\end{remark}
\noindent
%
The assumption on the invertibility of the matrix $S$ from (\ref{6p}) is crucial for the results in Proposition \ref{lem1} (iii), which 
states the connection between 
problems (\ref{al-22}) and (\ref{al-12}).
The next result 
states a new sufficient 
condition for the invertibility of $S$. \\
\begin{lemma}
	\label{lem1p}
If $m \leq n$ and the matrix $\tilde{A}$ of the initial system (\ref{al-11})
has full row rank, then $S$ is invertible.
\end{lemma}
\begin{proof}
According to the equality $S=W W^T$ (see (\ref{6p})) we get invertibility for 
$S = W W^T: q \times q$ (see (\ref{6p}))
if and only if the matrix $W^T: (n+q) \times  q$ has full column rank.
In this respect, let us suppose that $W^T z = 0$, for some $z \in \R^q$. 
As $\tilde{A}$ has full row rank, also the matrices $A$ from (\ref{1}) and $\bar{A}$ from (\ref{3}) will have full row rank, then (see e.g \cite{popabook})   
\be
\label{100}
\bar{A}^+ = \bar{A}^T (\bar{A} \bar{A}^T)^{-1}, ~~P= \bar{A}^+ \bar{A} = \bar{A}^T (\bar{A} \bar{A}^T)^{-1} \bar{A}.
\ee
Therefore, from (\ref{3}), (\ref{5}), (\ref{100}) and (\ref{5}) we obtain
$$
W^T z =0 \Leftrightarrow (I-P) Y^T z=0 \Leftrightarrow (I-P) \left[ 
		\begin{matrix}
		0 \\
		z \\ 
		\end{matrix}
		\right] = 0 \Leftrightarrow 
$$
$$
\left[ 
		\begin{matrix}
		0 \\
		z \\ 
		\end{matrix}
		\right] - \left[ 
		\begin{matrix}
		A^T \\
		\Gamma^T \\ 
		\end{matrix}
		\right] (A A^T + \Gamma \Gamma^T)^{-1} [ A ~~\Gamma ] \left[ 
		\begin{matrix}
		0 \\
		z \\ 
		\end{matrix}
		\right] = \left[ 
		\begin{matrix}
		0 \\
		0 \\ 
		\end{matrix}
		\right] \Leftrightarrow
$$
\be \label{101p}
	\left\{\begin{array}{c}
		A^T (A A^T + \Gamma \Gamma^T)^{-1} \Gamma z = 0 \\
		z - \Gamma^T (A A^T + \Gamma \Gamma^T)^{-1} \Gamma z = 0 \\
	\end{array} \right.
\ee
But, from our hypothesis the matrix $A^T$ has full column rank, thus from the first equation in (\ref{101p}) we get $\Gamma z = 0$, which gives us $z=0$ from the second equation and completes the proof.
\end{proof}  

{\bf Some comments on the structure of the matrix $S$.}

\vspace*{0.3cm}
According to Lemma \ref{lem1p}, if $\tilde{A}$ is underdetermined with full row rank, the matrix $S$ is invertible. In this case, we will also provide details on the structure of $S$ involving the orthogonal projections $P_i = P_{\cR(A^T_i}$ for a specific construction of the extended matrix $\bar{A}$ in (\ref{3}). Because both matrices ${A}, \bar{A}$ and the blocks ${A}_i, \bar{A}_i$ are underdetermined with full row rank the following are true (see e.g.~\cite{popabook}):
\be
\label{444}
\bar{P} = \bar{A}^+ \bar{A} = \sum_{i=1}^p \bar{P}_i, ~~~\bar{P}_i = 
 P_{\cR(\ba_i^T)} =  \ba^+_i \ba_i,
\ee
\be
\label{erlg13}
{P} = {A}^+ {A}, ~~~{P}_i = 
 P_{\cR(A_i^T)} =  A^+_i A_i,
\ee
\be
\label{erlg13p}
A^+ = A^T (A A^T)^{-1}, \bar{A}^+ = \bar{A}^T (\bar{A} \bar{A}^T)^{-1}, 
A_i^+ = A_i^T (A_i A_i^T)^{-1}, \bar{A}_i^+ = \bar{A}_i^T (\bar{A}_i \bar{A}_i^T)^{-1}.
\ee		
 In \cite{adumi} it is proposed the following construction of the matrix $\Gamma$ in (\ref{al-22})
\be
\label{erlg5}
\Gamma = D A, ~D=diag(I_{m_1}, -I_{m_2}, ..., (-1)^{p+1}I_{m_p}), ~\bar{A}_i = \begin{bmatrix} A_i & (-1)^{i+1}A_i \end{bmatrix}
\ee
\begin{proposition}
\label{prop1}
We can decompose the projector $\bar{P}$ depending on the projectors $P_i$ such as:
\begin{equation} \label{erlg6}
            \bar{P} = \begin{bmatrix}
                        \bar{P}_{11} & \bar{P}_{12}\\
                        \bar{P}_{21} & \bar{P}_{22}
                    \end{bmatrix}
\end{equation}
where $\bar{P}_{11} = \bar{P}_{22} = \frac{1}{2}\sum_i P_i$ and 
        $\bar{P}_{12} = \bar{P}_{21} = \frac{1}{2}\sum_i (-1)^{i+1}P_i$.
\end{proposition}
\begin{proof}
We will obtain  the expression of $\bar{P}_i$  depending on $P_i$ (see also (\ref{4}) and (\ref{erlg13p})). 
\begin{align} \label{erlg7}
            \begin{split}
                \bar{P}_i &= \mathcal{P}_{\mathcal{R}(\bar{A_i}^T)}\\
                &= \bar{A_i}^T(\bar{A_i}\bar{A_i}^T)^{-1}\bar{A_i}\\
                &= \begin{bmatrix}
                        A_i^T\\
                        (-1)^{i+1}A_i^T
                    \end{bmatrix}
                    \left(\begin{bmatrix}
                        A_i & (-1)^{i+1}A_i
                    \end{bmatrix}
                    \begin{bmatrix}
                        A_i^T\\
                        (-1)^{i+1}A_i^T
                    \end{bmatrix}\right)  ^{-1}
                    \begin{bmatrix}
                        A_i & (-1)^{i+1}A_i
                    \end{bmatrix}
                \\
                &= \begin{bmatrix}
                        A_i^T\\
                        (-1)^{i+1}A_i^T
                    \end{bmatrix}
                    (2A_iA_i^T)^{-1}
                    \begin{bmatrix}
                        A_i & (-1)^{i+1}A_i
                    \end{bmatrix}
                \\
                &= \frac{1}{2}\begin{bmatrix}
                        A_i^T(A_iA_i^T)^{-1}\\
                        (-1)^{i+1}A_i^T(A_iA_i^T)^{-1}
                    \end{bmatrix}
                    \begin{bmatrix}
                        A_i & (-1)^{i+1}A_i
                    \end{bmatrix}
                \\
                &= \frac{1}{2}\begin{bmatrix}
                        A_i^T(A_iA_i^T)^{-1}A_i & (-1)^{i+1}A_i^T(A_iA_i^T)^{-1}A_i\\
                        (-1)^{i+1}A_i^T(A_iA_i^T)^{-1}A_i & A_i^T(A_iA_i^T)^{-1}A_i
                    \end{bmatrix}
                \\
                &= \frac{1}{2}\begin{bmatrix}
                        P_i & (-1)^{i+1}P_i\\
                        (-1)^{i+1}P_i & P_i
                    \end{bmatrix}
            \end{split}
        \end{align} .
        The result is an expression of the projector $P$ split in 4 parts
        \begin{align} \label{erlg8}
            \begin{split}
                \bar{P} &= \sum_{i=1}^p \bar{P_i}\\
                &= \frac{1}{2}\begin{pmatrix}
                                \sum_{i=1}^p P_i & \sum_{i=1}^p (-1)^{i+1}P_i\\
                                \sum_{i=1}^p (-1)^{i+1}P_i & \sum_{i=1}^p P_i
                            \end{pmatrix}
            \end{split}
        \end{align}
				which completes the proof.
\end{proof}
The following formula can be useful when we have some additional information about the blocks $A_i$ in the reordered matrix $A$ from (\ref{1}).
\begin{proposition}
		\label{prop2}
We can decompose the submatrix $S$
depending on the elements from (\ref{erlg6}) - (\ref{erlg7}) as 
        \begin{equation} \label{erlg9}
            S = I_n - \bar{P}_{22} = I_n - \frac{1}{2}\sum_i P_i.
        \end{equation}
\end{proposition}
\begin{proof}
We can express S in terms of a restriction of W:
\begin{align} \label{erlg10}
            \begin{split}
               S &= Y(I_n - \bar{P})Y^T\\
               &= \begin{bmatrix}
                        0 & I_n
                    \end{bmatrix}
                    (I_n - \bar{P})
                    \begin{bmatrix}
                        0\\
                        I_n
                    \end{bmatrix}
                \\
                &= \begin{bmatrix}
                        0 & I_n
                    \end{bmatrix}
                    \begin{bmatrix}
                        0\\
                        I_n
                    \end{bmatrix}
                    -
                    \begin{bmatrix}
                        0 & I_n
                    \end{bmatrix}
                    \begin{bmatrix}
                        \bar{P}_{11} & \bar{P}_{12}\\
                        \bar{P}_{21} & \bar{P}_{22}
                    \end{bmatrix}
                    \begin{bmatrix}
                        0\\
                        I_n
                    \end{bmatrix}
                \\
                &= I_n -
                    \begin{bmatrix}
                        0 & I_n
                    \end{bmatrix}
                    \begin{bmatrix}
                        \bar{P}_{12}\\
                        \bar{P}_{22}
                    \end{bmatrix}
                \\
                &= I_n - \bar{P}_{22}.
            \end{split}
\end{align}
\end{proof}
%
%
\section{Full column rank overdetermined systems}
\label{inconsist}
In this section we will suppose that the initial matrix $\tilde{A}$ from (\ref{al-11}) is overdetermined, with full column rank, i.e.
\be
\label{alpha-21}
m ~\geq~ n, ~~~\rank(\tilde{A}) ~=~ n.
\ee
Unfortunately, in this case 
system (\ref{al-11}) is usually not consistent and must be reformulated  in the least squares sense: find $\tilde {x} \in \R^n$ such that 
\be
\label{ali-11}
\p \tilde{A} \tilde{x} - \tilde{b} \p = \min !
\ee 
{By analogy with section \ref{consist} we will consider the following augmentation scheme, but with respect to the least squares formulation of the corresponding steps. 
\be
\label{ali-12}
\p A x - b \p = \min ! ~~~\Rightarrow
\ee
\be
\label{ali-21}
\left \|  \left[ 
		\begin{matrix}
		A & 0 \\
		\Gamma & I_q \\
		\end{matrix}
		\right] \left[ 
		\begin{matrix}
		x \\
		y \\ 
		\end{matrix}
		\right] = \left[ 
		\begin{matrix}
		b \\
		0 \\ 
		\end{matrix}
		\right] \right \| = \min ! ~~~\Rightarrow
\ee
\be
\label{ali-22}
\left \| \left[ 
		\begin{matrix}
		A & B \\
		\Gamma & S \\
		\end{matrix}
		\right] \left[ 
		\begin{matrix}
		x \\
		y \\ 
		\end{matrix}
		\right] - \left[ 
		\begin{matrix}
		b \\
		f \\ 
		\end{matrix}
		\right] \right \| = \min !.
\ee
}
%
The matrix $A$ in (\ref{ali-12}) is constructed from $\tilde{A}$, but with respect to a block column structure, i.e.
$$
	A= [ A^1 ~A^2 \dots ~A^p ]:m \times n, ~A^i: m \times n_i, \sum_{i=1}^p n_i = n,
$$
\be
\label{fcr-1}
	A = P \tilde{A} Q, ~b = P \tilde{b},  ~P, Q ~{\rm orthogonal}.
\ee
In 
problem (\ref{ali-21}) 
the matrix $\Gamma$ has a block column structure 
$\Gamma = [ \Gamma^1  ~\Gamma^2 \dots \Gamma^p ]$
and is  constructed such that the augmented matrix 
\be
\label{fcr-3}
\bar{A} = \left[ 
		\begin{matrix}
		A \\
		\Gamma \\ 
		\end{matrix}
		\right] = [ \bar{A}^1 ~\bar{A}^2 \dots ~\bar{A}^p ], ~\bar{A}^i = \left[ 
		\begin{matrix}
		A^i \\
		\Gamma^i \\
		\end{matrix}
		\right], i=1, \dots, p
\ee
has mutually orthogonal column blocks, i.e.
\be
\label{fcr-4}
\bar{A}_i^T  \bar{A}^j = 0, ~\forall i \neq j.
\ee
But, because the block columns
$	\bar{A}$ and $Y = \left[ 
		\begin{matrix}
		0 \\
		I_q \\
		\end{matrix}
\right ] $ 
are not orthogonal, 
we consider the augmented problem (\ref{ali-22})
\be
\label{P3}
\left \| \left[ 
		\begin{matrix}
		A & B \\
		\Gamma & S \\
		\end{matrix}
		\right] \left[ 
		\begin{matrix}
		y \\
		 z \\
		\end{matrix}
		\right] - \left[ 
		\begin{matrix}
		b \\
		f \\
		\end{matrix}
		\right] \right \| = \min!,
\ee
with $B=m \times q, S: q \times q$ and $W=\left[ 
		\begin{matrix}
		B \\
		S \\
		\end{matrix}
		\right] : \bar{m} \times q$ such that
\be
		\label{1-4}
		\bar{A}^T W = 0.
		\ee
For this it suffices to define
\be
		\label{1-5}
		W = ( I_{\bar{m}} - \bar{P}) \left[ 
		\begin{matrix}
		0 \\
		I_q \\
		\end{matrix}
		\right] = ( I_{\bar{m}} - \bar{P}) Y, ~~{\rm with}~~
		Y= \left[ 
		\begin{matrix}
		0 \\
		I_q \\
		\end{matrix}
		\right], 
		~\bar{P}=P_{\cR(\bar{A})} = \bar{A} \bar{A}^+.
\ee
By using the well known relation $A A^+ A = A$ (see e.g. \cite{horn}) we obtain
\be
		\label{new2}
	W^T \bar{A} = Y^T (I_{\bar{m}} - \bar{P}) \bar{A} = 
		Y^T (\bar{A} - \bar{A} \bar{A}^+ \bar{A}) = Y^T (\bar{A} - \bar{A} ) = 0,
\ee
i.e.~(\ref{1-4}).

According to the initial problem (\ref{ali-11}) 
and the reordered one (\ref{ali-12}) the following result can be easily proved.
\begin{lemma}
		\label{lerlg}
The matrix $A$ is overdetermined and has full column rank.
If $\tilde{x}_{LS}, ~x_{LS}$, are the (unique)
minimal norm solutions of the problems (\ref{ali-11}) and (\ref{ali-12}), respectively, then
\be
		\label{erlg3}
		\tilde{x}_{LS} = Q x_{LS},
\ee
with $Q$ the orthogonal matrix from (\ref{fcr-1}).
\end{lemma}
The above lemma tells us that once the (unique) minimal norm solution of (\ref{ali-12}) is computed, we can easily obtain the similar solution of the initial system (\ref{ali-11}) through the equality in (\ref{erlg3}). Hence, as in section \ref{consist} we will show in the rest of this section how $x_{LS}$ can be computed through the minimal norm solution of the problem (\ref{ali-22}) for a particular choice of the vector $f$. Because the corresponding results are more elaborate than those from section  \ref{consist} we will first present them and at the end of the section we will give a solution scheme similar with the one from (1.1) - (1.4).

The next two results present
information about the problem (\ref{ali-22}) for a general right hand side vector $f$.
\begin{lemma}
\label{Plem2}
(i) We have the equalities
\be
\label{1-6}
\left[ \begin{matrix}
		A & B \\
		\Gamma & S \\
		\end{matrix}
		\right]^T \left[ 
		\begin{matrix}
		A & B \\
		\Gamma & S \\
		\end{matrix}
		\right] = \left[ 
		\begin{matrix}
		A^T A + \Gamma^T \Gamma & 0 \\
		0 & B^T B + S^T S \\
\end{matrix} \right],  
\ee
\be
\label{1-7}
S = W^T W, ~~B^T B = S - S^2,
\ee
\be
\label{1-88}
W = \left[ 
		\begin{matrix}
		-A(A^T A + \Gamma^T \Gamma)^{-1} \Gamma^T  \\
		I - \Gamma -A(A^T A + \Gamma^T \Gamma)^{-1} \Gamma^T  \\
		\end{matrix}
		\right]
\ee
(ii) The matrix $S$ is invertible.
\end{lemma}
\begin{proof}
From the orthogonality relations (\ref{1-4}) and (\ref{new2})  we obtain
$$
0 = A^T B + \Gamma^T S  = B^T A + S^T \Gamma,
$$ 
which gives us the equality (\ref{1-6}). From (\ref{1-5}) it results
\be
\label{1-9}
S=Y^T W = [ 0 ~I ] \left[ 
		\begin{matrix}
		B \\
		S \\
		\end{matrix}
		\right] = [ 0 ~I ] (I - \bar{P}) \left[ 
		\begin{matrix}
		0 \\
		I \\
		\end{matrix}
		\right] = Y^T (I - \bar{P}) Y.
\ee
Then, the first equality (\ref{1-7}) follows from (\ref{1-9}) and $W=(I - \bar{P}) Y$, whereas the second one from (\ref{1-7}) and $W=\left[ 
		\begin{matrix}
		B \\
		S \\
		\end{matrix}
		\right]$.\\
		As the matrix  $\bar{A} = \left[ 
		\begin{matrix}
		A \\
		\Gamma \\
		\end{matrix}
		\right]: \bar{m} \times n, ~\bar{m}=m+q \geq n$ from (\ref{fcr-3})  is overdetermined and has full column rank we have by successively using (\ref{1-5}) 
$$
		W = ( I_{\bar{m}} - \bar{P}) Y = ( I_{\bar{m}} - \bar{A} \bar{A}^+) Y = 
$$
$$
		( I_{\bar{m}} - \left[ 
		\begin{matrix}
		A \\
		\Gamma \\
		\end{matrix}
		\right] (A^T A + \Gamma^T \Gamma)^{-1} [ A^T ~\Gamma^T ] \left[ 
		\begin{matrix}
		0 \\
		I \\
		\end{matrix}
		\right] =  \left[ 
		\begin{matrix}
		-A (A^T A + \Gamma^T \Gamma)^{-1} \Gamma^T \\
		I_{\bar{m}} - \Gamma (A^T A + \Gamma^T \Gamma)^{-1} \Gamma^T \\
		\end{matrix}
		\right]
$$
from which (\ref{1-88}) holds.

(ii) If we apply the result in \cite{punt}, Theorem 5, eq. (5.3), page 121 to $W^T=E F$, with (see (\ref{1-5})) $W=(I_{\bar{m}} - \bar{P}) Y$, $E=Y^T, ~F=I_{\bar{m}} - \bar{P}$
we obtain 
$$
	\rank(W) = \rank(W^T) = \rank(Y^T (I - \bar{P})) = 
$$
\be
\label{new5}
\rank(Y^T) - 
\dim(\cR(Y) \cap \cN(I - \bar{P})) = 
q - \dim(\cR(Y) \cap \cR(\bar{A}).
\ee
Then, with similar arguments as in the proof of Lemma \ref{lem1p} and using  the first equality in (\ref{1-7}) we get that the matrix $S$ is invertible. 
\end{proof}
\begin{lemma}
\label{newlem1}
For any vector $f \in \R^q$ the matrix of the problem  (\ref{P3}) is overdetermined and has full column rank; moreover, its (unique) minimal norm solution
$\left[ 
		\begin{matrix}
		y \\
		 z\\
		\end{matrix}
		\right]_{LS}
$ is given by
\be  \label{new7}
	\left[ \begin{matrix}
		y \\
		 z\\
		\end{matrix}
		\right]_{LS} = \left[ 
		\begin{matrix}
		(A^T A + \Gamma^T \Gamma)^{-1} (A^T b + \Gamma^T f) \\
		 S^{-1} (B^T b + S f)\\
	\end{matrix} \right]
\ee
\end{lemma}
\begin{proof}
The column blocks 
$ \bar{A}=\left[ 
		\begin{matrix}
		A \\
		\Gamma\\
		\end{matrix}
		\right]$ and $W=\left[ 
		\begin{matrix}
		B \\
		S\\
\end{matrix} \right]$
have full column rank and are orthogonal (see (\ref{1-4})), which tell us that the problem matrix in (\ref{P3}), which is overdetermined (has dimensions $(m+q) \times (n+q), ~m > n$, has also full column rank. Hence, its minimal norm solution is the unique solution of the associated normal equation (see also (\ref{1-6}) and (\ref{1-7}))
\be
		\label{fcr-5}
		\left[ 
		\begin{matrix}
		A^T A + \Gamma^T \Gamma & 0 \\
		 0 & S \\
		\end{matrix}
		\right] \left[ 
		\begin{matrix}
		y \\
		 z\\
		\end{matrix}
		\right]_{LS} = 
		\left[ 
		\begin{matrix}
		A^T b + \Gamma^T f \\
		 B^T b + S f\\
		\end{matrix}
		\right]
\ee
which gives us (\ref{new7}) and completes the proof.
\end{proof}
Starting with the next result a special choice will be made on the vector $f$. This assumption will be kept in the rest of the section. 
\begin{lemma}
\label{newlem2}
If $\left[ 
		\begin{matrix}
		y \\
		 z\\
		\end{matrix}
		\right]_{LS}$ is the minimal norm solution of the problem (\ref{P3}) then 
\be
\label{P0-1}
\Gamma y + S z = 0 ~~\Leftrightarrow~~ f = 0.
\ee
\end{lemma}
\begin{proof}
From (\ref{1-88}) it results
\be
\label{P0-2}
W = \left[ 
		\begin{matrix}
		B \\
		S \\
		\end{matrix}
		\right] =  \left[ 
		\begin{matrix}
		 - A (A^T A + \Gamma^T \Gamma)^{-1} \Gamma^T \\
		 I - \Gamma (A^T A + \Gamma^T \Gamma)^{-1} \Gamma^T\\
		\end{matrix}
		\right],
\ee
therefore
\be
\label{P0-3}
B = - A (A^T A + \Gamma^T \Gamma)^{-1} \Gamma^T = -A G \Gamma^T, ~{\rm with}~ G = (A^T A + \Gamma^T \Gamma)^{-1}.
\ee
If $\left[ 
		\begin{matrix}
		y \\
		 z\\
		\end{matrix}
		\right]_{LS}$  is the minimal norm solution of the problem (\ref{P3}), then $y$ and $z$ are given by (\ref{new7}), therefore
$$
\Gamma y + S z = \Gamma G A^T b + \Gamma G \Gamma^T f + B^T b + S f = 
$$
\be
\label{P0-4}
\Gamma G A^T b + \Gamma G \Gamma^T f - \Gamma G A^T b + Sf = 
(\Gamma G \Gamma^T + S) f.
\ee
This equality gives us the conclusion of the lemma, because the matrix $\Gamma G \Gamma^T + S$ is symmetric and positive definite.
\end{proof}
For the special choice of $f$ from (\ref{P0-1})  the problem (\ref{P3}) can be written as 
$$
\min_{(y, z) \in \R^{n+q}} \left\|  \left[ 
		\begin{matrix}
		A & B \\
		\Gamma & S \\
		\end{matrix}
		\right] \left[ 
		\begin{matrix}
		y \\
		 z \\
		\end{matrix}
		\right] - \left[ 
		\begin{matrix}
		b \\
		0 \\
		\end{matrix}
		\right] \right\|^2  \Leftrightarrow 
$$
\be
\label{P4}
		\min_{(y, z) \in \R^{n+q}} 
		(\p Ay + Bz - b\p^2 + \p \Gamma y + Sz \p^2).
\ee
Moreover, the matrix 
\be
\label{1000}
\Omega = I+(A^T A + \Gamma^T \Gamma)^{-1} \Gamma^T S^{-1} \Gamma 
\ee
is similar to the matrix 
\be
\label{1001}
H=I + (A^T A + \Gamma^T \Gamma)^{-\frac{1}{2}} \Gamma^T S^{-1} \Gamma (A^T A + \Gamma^T \Gamma)^{-\frac{1}{2}}
\ee
 because
\be
\label{1002}
H = (A^T A + \Gamma^T \Gamma)^{\frac{1}{2}} ~~\Omega~~ (A^T A + \Gamma^T \Gamma)^{-\frac{1}{2}}.
\ee
But, $H$ is SPD, therefore $\Omega$ is invertible. We introduce the notation
\be
\label{0-5}
\phi(y, z) = \left \| \left[ 
		\begin{matrix}
		A & B \\
		\Gamma & S \\
		\end{matrix}
		\right] \left[ 
		\begin{matrix}
		y \\
		 z \\
		\end{matrix}
		\right] - \left[ 
		\begin{matrix}
		b \\
		0 \\
		\end{matrix}
		\right] \right \|^2.
\ee
\begin{lemma}
\label{newlem3}
We have the inequality
\be
\label{0-6}
\phi(y, z) ~\geq~ \min \{ \p Ax - b \p^2, x \in \R^n \} = \p P_{\cN(A^T)}(b) \p^2, ~\forall (y, z).
\ee
\end{lemma}
\begin{proof}
From (\ref{P4}) and (\ref{P0-3}) we get
$$
\phi(y, z) ~\geq~ \p Ay + Bz - b \p^2 = \p Ay - A G \Gamma^T z - b \p^2 =
$$
\be
\label{0-7}
\p A(y -  G \Gamma^T z) - b \p^2 ~\geq~ \min \{ \p Ax - b \p^2, x \in \R^n \} = \p P_{\cN(A^T)}(b) \p^2,
\ee
and the proof is complete.
\end{proof}
\begin{lemma}
\label{newlem4}
Let  $x_{LS}$ be the (unique) solution of the problem (\ref{ali-12}), $\Omega$ the (invertible) matrix from (\ref{1000}) and $y, z$ defined by 
\be
\label{0-8}
y = \Omega^{-1} x_{LS}, ~~~z = -S^{-1} \Gamma y.
\ee
Then
\be
\label{0-9}
\phi(y, z) = \p Ax_{LS} - b \p^2 = \p P_{\cN(A^T)}(b) \p^2.
\ee
\end{lemma}
\begin{proof}
For $\left[ 
		\begin{matrix}
		y \\
		z \\
		\end{matrix}
		\right]$, with $y, z$ as in (\ref{0-8}) we successively get
$$
\left \| \left[ 
		\begin{matrix}
		A & B \\
		\Gamma & S \\
		\end{matrix}
		\right] \left[ 
		\begin{matrix}
		y \\
		 z \\
		\end{matrix}
		\right] - \left[ 
		\begin{matrix}
		b \\
		0 \\
		\end{matrix}
		\right] \right \|^2 = \p Ay + Bz - b\p^2 + \p \Gamma y + Sz \p^2 = 
$$
$$
		\p Ay -A (A^T A + \Gamma^T \Gamma)^{-1} \Gamma^T z - b \p^2
		 + 0 = 
$$
$$
		\p Ay -A (A^T A + \Gamma^T \Gamma)^{-1} \Gamma^T (-S^{-1} \Gamma y) - b \p^2
$$
$$
		\p A(I + (A^T A + \Gamma^T \Gamma)^{-1} \Gamma^T S^{-1} \Gamma) y - b \p^2 = 
$$
\be \label{P6}
	\p Ax_{LS} - b \p^2  = 
		\p  P_{\cN(A^T)}(b) \p^2.
\ee
\end{proof}
From the above lemmas it results that 
$\left[ 
\begin{matrix}
		y \\
		z \\
		\end{matrix}
\right]_{LS}$
is the (unique) minimal norm solution of the problem (\ref{ali-22})
with $f=0$ (see also (\ref{P4})) if and only if
\be
\label{0-10}
\left \| \left[ 
		\begin{matrix}
		A & B \\
		\Gamma & S \\
		\end{matrix}
		\right] \left[ 
		\begin{matrix}
		y \\
		 z \\
		\end{matrix}
		\right] - \left[ 
		\begin{matrix}
		b \\
		0 \\
		\end{matrix}
		\right] \right\|^2 = \p P_{\cN(A^T)}(b) \p^2.
\ee
Therefore we can conclude with the following result, which gives us a  direct connection between the augmented  problem (\ref{ali-22}) with $f=0$  and  (\ref{ali-12}).
\begin{lemma}
\label{newlem5}
 The following properties hold.\\
(i) Let $x_{LS}$ be the minimal norm solution of the problem (\ref{ali-12}) and $y, z$ defined by (\ref{0-8}). Then $\left[ 
		\begin{matrix}
		y \\
		z \\
		\end{matrix}
		\right]_{LS}$ is the minimal norm solution of the problem (\ref{ali-22});\\
(ii) Let $\left[ 
		\begin{matrix}
		y \\
		z \\
		\end{matrix}
		\right]_{LS}$ be the minimal norm solution of the problem (\ref{ali-22}). Then $z = -S^{-1} \Gamma y$ and $x_{LS}$ given by 
\be
\label{0-11}
x_{LS} = \Omega y = \left ( I + (\bar{A}^T \bar{A})^{-1} \Gamma^T S^{-1} \Gamma \right ) y
\ee
is the minimal norm solution of the problem (\ref{ali-12}).\\
(iii) The minimal norm solution of the problem (\ref{ali-22}), $\left[ 
		\begin{matrix}
		y \\
		z \\
		\end{matrix}
		\right]_{LS}$ is given by
\be
\label{eqfinal-1}
\left[ 
		\begin{matrix}
		y \\
		z \\
		\end{matrix}
		\right]_{LS} =  \left[ 
		\begin{matrix}
		(\bar{A}^T \bar{A})^{-1} A^T b \\
		-S^{-1} \Gamma  (\bar{A}^T \bar{A})^{-1} A^T b \\
		\end{matrix}
		\right]  
\ee
with
\be
\label{eqfinal-2}
(\bar{A}^T \bar{A})^{-1} = \left[ 
		\begin{matrix}
		(\bar{A}_1^T \bar{A}_1)^{-1} & 0 & 0 & \dots & 0 \\
		0 & (\bar{A}_2^T \bar{A}_2)^{-1} & 0 & \dots & 0 \\
		\dots & \dots & \dots & \dots & \dots \\
		0 & 0 & 0 & \dots & (\bar{A}_p^T \bar{A}_p)^{-1} \\
		\end{matrix}
		\right] 
		\ee
\end{lemma}
\begin{proof}
(i) It results directly from (\ref{newlem4}) and (\ref{0-10}).\\
(ii) Because $f=0$, according to Lemma \ref{newlem2} we get $\Gamma y + S z = 0$. Hence $z = -S^{-1} \Gamma y$. Moreover, replaying the calculations in (\ref{P6}) and also using (\ref{0-10}) and (\ref{0-11}) we obtain
$$
\p P_{\cN(A^T)}(b) \p^2 = \left \| \left[ 
		\begin{matrix}
		A & B \\
		\Gamma & S \\
		\end{matrix}
		\right] \left[ 
		\begin{matrix}
		y \\
		 z \\
		\end{matrix}
		\right] - \left[ 
		\begin{matrix}
		b \\
		0 \\
		\end{matrix}
		\right] \right \|^2 = 
		$$
		$$
		\p A(I + (A^T A + \Gamma^T \Gamma)^{-1} \Gamma^T S^{-1} \Gamma) y - b \p^2 = \p A \Omega y - b \p^2 = 
		\p A x - b \p^2,
$$
which tells us that $x$ is the minimal norm solution of (\ref{ali-12}) (see also (\ref{1000})).\\
(iii) Equation (\ref{eqfinal-1}) holds from the assumption $f=0$, (\ref{new7}), (\ref{P0-3}) and equation (\ref{eqfinal-2}) from (\ref{fcr-4}).
\end{proof}

\paragraph{Solution Procedure.}
The algorithm proposed  to compute a solution of the initial system (\ref{ali-11}) through the minimal norm solution of the augmented system (\ref{ali-22}) (with $f=0$) is the following.
		\begin{itemize}
		\item[(2.1)] The minimal norm solution $\left[ 
		\begin{matrix}
		y \\
		z \\
		\end{matrix}
		\right]_{LS}$ of the problem (\ref{ali-22}) is computed (in parallel) by (\ref{eqfinal-1}) - (\ref{eqfinal-2}).
		\item[(2.2)] The minimal norm solution $x_{LS}$ of the problem (\ref{ali-12}) is computed (in parallel) by (\ref{0-11}).
		\item[(2.3)] The minimal norm solution $\tilde{x}_{LS}$ of the  initial problem (\ref{ali-11}) is computed ( by only a permutation of components) by (\ref{erlg3}).
		\end{itemize}
		
%

\section{Conclusions}
The article has provided insight into projection methods that are applied to augmented systems
for both the underdetermined case as well as for overdetermined systems.
These results are the basis for extensions to make the methods relevant for practice.
Such techniques must exploit parallel computing, a topic outside the scope of this work.

The augmentation should  be problem specific. In
the case that the system originates from a discretization of a PDE,
an efficient augmentation can be derived from a domain decomposition.
Other interesting extensions include relaxing the strict orthogonality between
the augmented rows (or columns) to an only approximate orthogonality.
The resulting method will then 
not be a direct solver, but an iterative one. The study of such methods
is left to future work that can be based on the results of this article.

\end{document}